\pdfoutput=1
\documentclass[11pt,oneside]{amsart}
\usepackage[style=alphabetic,maxalphanames=4]{biblatex}
\usepackage{todonotes}
\usepackage[paperheight=279mm,paperwidth=18cm,textheight=26cm,textwidth=14cm,includehead]{geometry}
\usepackage[T1]{fontenc}
\usepackage[utf8]{inputenc}
\usepackage{amsfonts,amssymb,amsmath,amsthm}
\usepackage{mathtools}
\usepackage{etoolbox}
\usepackage{xcolor} 

\usepackage[unicode]{hyperref}
\hypersetup{
bookmarks=true,
colorlinks=true,
citecolor=[rgb]{0,0,0.5},
linkcolor=[rgb]{0,0,0.5},
urlcolor=[rgb]{0,0,0.75},
pdfpagemode=UseNone,
pdfstartview=FitH,
pdfdisplaydoctitle=true,
pdflang=en-US
}

\def\PZdefchar#1{
  \expandafter\def\csname frak#1\endcsname{\mathfrak{#1}}
  \expandafter\def\csname bf#1\endcsname{\mathbf{#1}}
  \expandafter\def\csname scr#1\endcsname{\mathcal{#1}}
  \expandafter\def\csname cal#1\endcsname{\mathcal{#1}}}
\def\PZdefloop#1{\ifx#1\PZdefloop\else\PZdefchar#1\expandafter\PZdefloop\fi}
\PZdefloop abcdefghijklmnopqrstuvwxyzABCDEFGHIJKLMNOPQRSTUVWXYZ\PZdefloop

\newcommand{\R}{\mathbb{R}}

\newcommand{\Z}{\mathbb{Z}}

\newcommand{\one}{\mathbf{1}}

\numberwithin{equation}{section}
\newtheorem{theorem}{Theorem}
\numberwithin{theorem}{section}
\newtheorem{proposition}[theorem]{Proposition}
\newtheorem{lemma}[theorem]{Lemma}
\newtheorem{corollary}[theorem]{Corollary}
\theoremstyle{definition}

\newtheorem{notation}[theorem]{Notation}

\theoremstyle{remark}
\newtheorem{remark}[theorem]{Remark}
\newtheorem{example}[theorem]{Example}

\DeclarePairedDelimiter\abs{\lvert}{\rvert}

\DeclarePairedDelimiter\norm{\lVert}{\rVert}

\providecommand\given{}
\newcommand\SetSymbol[1][]{%
\nonscript\:#1\vert
\allowbreak
\nonscript\:
\mathopen{}}
\DeclarePairedDelimiterX\Set[1]\{\}{\renewcommand\given{\SetSymbol[\delimsize]}#1}
\DeclarePairedDelimiterXPP\EE[1]{\E}{\lparen}{\rparen}{}{\renewcommand\given{\SetSymbol[\delimsize]}#1} 

\makeatletter
\newcommand\@avsum[2]{%
  {\sbox0{$\m@th#1\sum$}%
   \vphantom{\usebox0}%
   \ooalign{%
     \hidewidth
     \smash{\vrule height\dimexpr\ht0+1pt\relax depth\dimexpr\dp0+1pt\relax}%
     \hidewidth\cr
     $\m@th#1\sum$\cr
   }%
  }%
}
\newcommand{\avsum}{\mathop{\mathpalette\@avsum\relax}\displaylimits}
\newcommand\@avprod[2]{%
  {\sbox0{$\m@th#1\prod$}%
   \vphantom{\usebox0}%
   \ooalign{%
     \hidewidth
     \smash{\vrule height\dimexpr\ht0+1pt\relax depth\dimexpr\dp0+1pt\relax}%
     \hidewidth\cr
     $\m@th#1\prod$\cr
   }%
  }%
}
\newcommand{\avprod}{\mathop{\mathpalette\@avprod\relax}\displaylimits}
\newcommand{\@avL}[2]{%
\ooalign{{$\m@th#1\mbox{--}$}\cr {$\m@th#1 L$}\cr}}
\newcommand{\avL}{\mathpalette\@avL\relax}
\newcommand{\@avell}[2]{%
\ooalign{{$\m@th#1\mbox{--}$}\cr {$\m@th#1 \ell$}\cr}}
\newcommand{\avell}{\mathpalette\@avell\relax}
\newcommand{\@avD}{%
  \ooalign{{$\mathrm{D}$}\cr \hidewidth\raise.2ex\hbox{$\vert$}\hidewidth\cr}}
\newcommand{\avDec}{\@avD\mathrm{ec}}
\makeatother
\newcommand{\Dec}{\mathcal{D}} 

\DeclareMathOperator{\lin}{lin}

\DeclareMathOperator{\supp}{supp}

\newcommand{\Part}[2][]{\calP(\ifstrempty{#1}{}{#1,}#2)} 
\newcommand{\unitint}{{[0,1]}}
\newcommand{\unitsquare}{{[0,1]^2}}
\newcommand{\unitcube}{{[0,1]^3}}
\hypersetup{
pdftitle={Decoupling for two quadratic forms in three variables},
pdfauthor={Guo, Oh, Roos, Yung, Zorin-Kranich},
}

\def\lesim{\lesssim}
\def\mc{\mathcal}
\def\bt{\mathbf{t}}

\begin{document}

\title[Decoupling for two quadratic forms in three variables]{Decoupling for two quadratic forms in three variables: a complete characterization}
\author[S.~Guo]{Shaoming Guo}
\address[SG]{Department of Mathematics\\ University of Wisconsin-Madison}
\email{shaomingguo@math.wisc.edu}
\author[C.~Oh]{Changkeun Oh}
\address[CO]{Department of Mathematics\\ University of Wisconsin-Madison}
\email{coh28@wisc.edu}
\author[J.~Roos]{Joris Roos}
\address[JR]{Department of Mathematics\\ University of Wisconsin-Madison}
\email{jroos.math@gmail.com}
\curraddr{Department of Mathematical Sciences \\ University of Massachusetts Lowell}
\author[P.-L.~Yung]{Po-Lam Yung}
\address[PY]{Department of Mathematics\\ The Chinese University of Hong Kong
  \newline and Mathematical Sciences Institute\\ The Australian National University}
\email{plyung@math.cuhk.edu.hk}
\email{polam.yung@anu.edu.au}
\author[P.~Zorin-Kranich]{Pavel Zorin-Kranich}
\address[PZ]{Mathematical Institute\\ University of Bonn}
\email{pzorin@math.uni-bonn.de}

\begin{abstract}
We prove sharp decoupling inequalities for all degenerate surfaces of codimension two in $\mathbb{R}^5$ given by two quadratic forms in three variables.
Together with previous work by Demeter, Guo, and Shi in the non-degenerate case, this provides a classification of decoupling inequalities for pairs of quadratic forms in three variables.
\end{abstract}

\maketitle

\section{Introduction}

We begin by recalling the definition of decoupling constants.
Let $d, n\ge 1$ be integers.
For real quadratic forms $Q_1, \dotsc, Q_n$ in $d$ variables, consider the surface
\begin{equation}
\label{eq:S0}
\mc{S}_0
=
\Set{ (\bt, Q_1(\bt), \dotsc, Q_n(\bt)) \given \bt\in [0, 1]^d }.
\end{equation}
For a dyadic cube $\Box\subset [0, 1]^d$ with side length $\delta$, we will use $f_{\Box}$ to denote a function with
\begin{equation}
\label{eq:fourier-supp}
\supp(\widehat{f_{\Box}})
\subset
\Set{ (\bt, Q_1(\bt)+\delta^{(1)}, \dots, Q_n(\bt)+\delta^{(n)}) \given \bt\in \Box, \abs{\delta^{(1)}},\dotsc,\abs{\delta^{(n)}}\le \delta^2 }.
\end{equation}
For $2 \leq p < \infty$ and a dyadic number $\delta \in (0,1)$, the \emph{decoupling constant} $\Dec_{\mc{S}_0}(\delta, p)$ is the smallest constant $A$ such that the inequality
\begin{equation}
\label{eq:dec-const}
\norm[\Big]{\sum_{\Box \in \Part{\delta}} f_{\Box} }_{L^p(\R^{d+n})}
\le A
\Big(\sum_{\Box \in \Part{\delta}} \norm{f_{\Box}}_{L^p(\R^{d+n})}^p \Big)^{1/p},
\end{equation}
where $\Part{\delta}$ is the partition of $[0,1]^{d}$ into dyadic cubes with side length $\delta$, holds for every choice of functions $f_{\Box}$ satisfying \eqref{eq:fourier-supp}; replacing the $\ell^p L^p$ norm on the right hand side of \eqref{eq:dec-const} by $\ell^{\infty} L^{\infty}$ gives the definition when $p = \infty$.

In this article we are interested in the case $d=3$, $n=2$.  We will also use existing results for smaller values of $d$ and $n$, which necessitates defining \eqref{eq:dec-const} in more generality.
We will denote by $(P,Q)$ a pair of real quadratic forms in three variables and by $\mc{S}$ the surface
\begin{equation}\label{S_def}
\mc{S} := \Set{ (r, s, t, P(r, s, t), Q(r, s, t)) \given (r, s, t)\in [0, 1]^3 }.
\end{equation}
Our goal is to prove, for every $2 \leq p < \infty$, an essentially sharp bound on $\Dec_{P,Q}(\delta,p) := \Dec_{\mc{S}}(\delta, p)$ as $\delta\to 0$.
In order to formulate our results in a concise way, we introduce the \emph{sharp decoupling exponent}
\begin{equation}
\gamma_{P,Q}(p) := \inf \Set{ \gamma\geq 0 \given \Dec_{P,Q}(\delta,p) \lesssim \delta^{-\gamma} }.
\end{equation}
In other words, $\gamma_{P,Q}(p)$ is the smallest exponent $\gamma$ such that, for every $\epsilon>0$, we have
\begin{equation}
\label{eq:dec-ineq}
\Dec_{P,Q}(\delta,p) \lesssim_{\epsilon} \delta^{-\gamma-\epsilon}
\text{ for every } \delta \in (0,1).
\end{equation}

\subsection{Previous results}
The case of linearly dependent $P$ and $Q$ is equivalent to the case $n=1$ of \eqref{eq:S0}.
In this case, sharp decoupling inequalities were proved by Bourgain and Demeter \cite[Theorem 1.1]{MR3736493}.
Henceforth, we assume that $P$ and $Q$ are linearly independent.

Moreover, we assume that there is no linear change of variables in $(r, s, t)$ such that $P$ and $Q$ both omit one of the variables, as otherwise we can reduce to the case $d=n=2$ that was considered by Bourgain and Demeter in \cite{BD16}.

We say that the pair $(P,Q)$ is \emph{non-degenerate} if both of the following conditions hold:
\begin{equation}
\label{eq:det-condition}
\det( \nabla P,\nabla Q ,u) \not\equiv 0 \text{ for all } u \in \R^3\setminus\{0\},
\end{equation}
\begin{equation}
\label{eq:no-common-factor}
\text{$P$ and $Q$ do not share a common linear real factor.}
\end{equation}

Examples show that, for any pair $(P,Q)$ of quadratic polynomials in $3$ variables, the sharp decoupling exponent satisfies
\begin{equation}
\label{non-deg-decexp}
\gamma_{P,Q}(p) \geq
\begin{cases}
3(\frac 1 2-\frac 1 p) & \text{ if } 2\le p\le 14/3,\\
3-\frac{10}p & \text{ if } 14/3\le p\le \infty,
\end{cases}
\end{equation}
see \eqref{eq:16}.
In the non-degenerate case, Demeter, Shi, and the first author \cite{MR3945730} (see also \cite{GZo19} for a simplified proof) proved that, in fact,
\begin{equation}\label{DGS-decexp}
\gamma_{P,Q}(p) =
\begin{cases}
3(\frac 1 2-\frac 1 p) & \text{ if } 2\le p\le 14/3,\\
3-\frac{10}p & \text{ if } 14/3\le p\le \infty.
\end{cases}
\end{equation}
Therefore, it is a natural question to find the minimal requirements for $P$ and $Q$ such that the decoupling inequality \eqref{eq:dec-ineq} holds with the smallest possible sharp decoupling exponent \eqref{DGS-decexp}.

It was pointed out in \cite{MR3945730} that \eqref{eq:no-common-factor} is a necessary condition for \eqref{DGS-decexp} to hold: If $P$ and $Q$ share a common linear factor, then the surface $\mc{S}$ given by \eqref{S_def} is flat on the hyperplane determined by that common linear factor, and therefore we cannot expect any satisfactory decoupling inequality like \eqref{DGS-decexp}.
For a more precise argument, see \eqref{eq:20}.

It was also observed in \cite[Appendix]{MR3945730} that \eqref{eq:det-condition} is necessary for the multilinear approach to proving \eqref{DGS-decexp}, which originates in \cite{MR3374964,MR3548534}, to work.
More precisely, if \eqref{DGS-decexp} fails, then no collection of tangent spaces to the surface $\mc{S}$ can satisfy the transversality conditions in \cite[Theorem 1.15, (c)]{MR2377493} that characterize validity of Brascamp--Lieb inequalities, on which the multilinear approach relies crucially.
However, the question whether the assumption \eqref{eq:det-condition} is necessary for \eqref{DGS-decexp} was left open. 

In the current paper, we will give an affirmative answer to the above question, that is, we will prove that \eqref{DGS-decexp} holds if and only if $P$ and $Q$ satisfy \eqref{eq:det-condition} and \eqref{eq:no-common-factor}.
To this end, we will find the sharp decoupling exponents for every degenerate pair $(P, Q)$.
We will see that, for any degenerate pair $(P, Q)$, the sharp decoupling exponents are strictly larger than \eqref{DGS-decexp} in some range of $p$'s.

\subsection{Classification of pairs $(P,Q)$}
We say that two pairs of quadratic forms $(P, Q)$ and $(P', Q')$ are equivalent if there exist two invertible real matrices $L_1\in M_{3\times 3}$ and $L_2\in M_{2\times 2}$ such that
\begin{equation}
\label{L1L2}
(P'(r, s, t), Q'(r, s, t))
= L_2\cdot (P(L_1\cdot (r, s, t)), Q(L_1\cdot (r, s, t))).
\end{equation}
This defines an equivalence relation, which we denote by
\begin{equation}
(P, Q)\equiv (P', Q').
\end{equation}
By changing coordinates, it is easy to see that
\begin{equation}
\Dec_{P, Q}(\delta, p)\approx \Dec_{P', Q'}(\delta, p),
\end{equation}
with an implicit constant depending only on $L_{1}$ and $L_{2}$ in \eqref{L1L2}, and in particular $\gamma_{P,Q}(p) = \gamma_{P',Q'}(p)$.
The following result describes all possible sharp decoupling exponents for two quadratic forms in three variables that do not omit any variable.

\begin{theorem}
\label{thm:classification+exponents}
Let $(P, Q)$ be a pair of linearly independent quadratic forms.
Assume that there is no linear change of variables in $(r, s, t)$ after which $P$ and $Q$ both omit one of the variables.
Then exactly one of the following alternatives holds.
\begin{enumerate}
\item\label{it:class:non-deg} $(P,Q)$ is non-degenerate, that is, both \eqref{eq:det-condition} and \eqref{eq:no-common-factor} hold.
In this case, the sharp decoupling exponent is given by \eqref{DGS-decexp}.
\item\label{it:class:common-factor} \eqref{eq:det-condition} holds, but \eqref{eq:no-common-factor} fails.
In this case, $(P, Q) \equiv (rs, rt)$, and
\begin{equation}
\label{rsrt_decexp}
\gamma_{P,Q}(p)
=
\begin{cases}
2-\frac 4 p & \text{ if } 2\le p\le 6,\\
3-\frac{10}p & \text{ if } 6\le p\le \infty.
\end{cases}
\end{equation}
\item\label{it:complete-square} \eqref{eq:det-condition} fails, but \eqref{eq:no-common-factor} holds.
In this case, either $(P,Q) \equiv (r^2, s^2\pm t^2)$, or $(P, Q) \equiv (r^2, s^2+rt)$.
In both subcases,
\begin{equation}
\label{r2s2rt_decexp}
\gamma_{P,Q}(p) =
\begin{cases}
3(\frac 1 2-\frac 1 p) & \text{ if } 2\le p\le 4,\\
\frac{5}{2}-\frac{7}{p} & \text{ if } 4\le p\le 6,\\
3-\frac{10}p & \text{ if } 6\le p\le \infty.
\end{cases}
\end{equation}
\end{enumerate}
\end{theorem}

Theorem~\ref{thm:classification+exponents} combines several results.
Our main result is the bound $\leq$ in \eqref{r2s2rt_decexp} in the case $(P,Q) \equiv (r^{2},s^{2}+rt)$, which we repeat in Theorem~\ref{thm:main} and discuss in more detail below.

The classification of pairs of quadratic forms is the content of Proposition~\ref{prop:classification}.
The upper bound $\leq$ in \eqref{rsrt_decexp} in the case $(P,Q) \equiv (rs,rt)$ is the content of Proposition~\ref{190707prop3}.
The upper bound $\leq$ in \eqref{r2s2rt_decexp} in the case $(P,Q) \equiv (r^{2}, s^{2}\pm t^{2})$ follows directly from the corresponding inequalities for the parabola $(r,r^{2})$, see \cite{MR3374964}, and the surfaces $(s,t,s^{2}\pm t^{2})$, see \cite[Theorem 1.1]{MR3736493}.
Finally, examples that show the lower bounds $\geq$ in \eqref{DGS-decexp}, \eqref{rsrt_decexp}, and \eqref{r2s2rt_decexp} are discussed in Section~\ref{sec:lower-bd}.

For a pair of linearly independent quadratic forms $(P,Q)$ in three variables that omit at least one variable (possibly after a linear change of variables), the sharp decoupling exponent is also given by \eqref{rsrt_decexp}.
The upper bound follows from flat decoupling \eqref{eq:flat-dec} and the decoupling inequality for two quadratic forms in two variables that was proved in \cite{MR3447712}, similarly to the proof of Proposition~\ref{190707prop3}.
The lower bound follows from Proposition~\ref{prop:skew-lower-bd} with $d'=n'=2$ when $2 \leq p \leq 6$ (and from \eqref{non-deg-decexp} when $6 \leq p \leq \infty$).

\subsection{The main decoupling inequality}
Let us state the main new part of Theorem~\ref{thm:classification+exponents} more explicitly.
\begin{theorem}
\label{thm:main}
Let $\mc{S}$ be the surface given by $(r, s, t, r^2, s^2+rt)$.
Then, for every $\epsilon>0$, we have
\begin{equation}
\label{r2s2rt_dec_const}
\Dec_{\mc{S}}(\delta, p) \lesim_{\epsilon, p}
\begin{cases}
\delta^{-3(\frac 1 2-\frac 1 p)-\epsilon} & \text{ if } 2\le p\le 4,\\
\delta^{-(\frac{5}{2}-\frac{7}{p})-\epsilon} & \text{ if } 4\le p\le 6,\\
\delta^{-(3-\frac{10}p)-\epsilon} & \text{ if } 6\le p\le \infty.
\end{cases}
\end{equation}
\end{theorem}
It is well-known that, for an integer $s\geq 1$, the study of the decoupling constant $\Dec_{\mc{S}}(\delta, p)$ with $p=2s$ is closely related to the problem of counting integer solutions to the Diophantine system
\begin{equation}\label{diophantine}
\begin{split}
 x_1+\dots+x_s& =x_{s+1}+\dots+x_{2s},\\
 y_1+\dots+y_s&=y_{s+1}+\dots+y_{2s},\\
 z_1+\dots+z_s&=z_{s+1}+\dots+z_{2s},\\
 P(x_1, y_1, z_1)+\dots+P(x_s, y_s, z_s)&=P(x_{s+1}, y_{s+1}, z_{s+1})+\dots+P(x_{2s}, y_{2s}, z_{2s}),\\
 Q(x_1, y_1, z_1)+\dots+Q(x_s, y_s, z_s)&=Q(x_{s+1}, y_{s+1}, z_{s+1})+\dots+Q(x_{2s}, y_{2s}, z_{2s}).
\end{split}
\end{equation}
Indeed, let $J_{\calS, s}(N)$ denote the number of integral solutions of \eqref{diophantine}, where all variables $x_i, y_i, z_i$ with $1\le i\le 2s$ take values in $\Set{0, 1, \dotsc, N}$.
Then, by the argument in \cite[Corollary 4.2]{MR3548534}, we have
\begin{equation}
J_{\calS, s}(N) \lesssim N^{3} \Dec_{\calS}(N^{-1},2s)^{2s}.
\end{equation}
Theorem~\ref{thm:classification+exponents} implies sharp estimates on $J_{\mc{S}, s}(N)$ for every $N$ and every $s\ge 1$.
For instance, if we take $P=r^2$ and $Q=s^2+rt$, then Theorem~\ref{thm:main} implies that 
\begin{equation}
J_{\mc{S}, s}(N)\lesim_{s, \epsilon} N^{3s+\epsilon}+N^{5s-4+\epsilon}+N^{6s-7+\epsilon},
\end{equation}
for every $\epsilon>0$.
In particular, when $s=2$ (which corresponds to $p=4$), we have that $J_{\mc{S}, 2}(N)\lesim_{\epsilon} N^{6+\epsilon}$.
Notice that if we set $x_i=x_{i+2}, y_i=y_{i+2}, z_i=z_{i+2}$ for every $i=1, 2$, then we obtain a trivial lower bound $J_{\mc{S}, 2}(N)\ge N^6$.
In this sense, the number of integral solutions to the system \eqref{diophantine} still shows diagonal behavior when $s = 2$.

In Section~\ref{section_counting}, we will present a simple direct proof of the bound $J_{\mc{S}, 2}(N)\lesim_{\epsilon} N^{6+\epsilon}$ that relies on elementary counting methods, rather than decoupling inequalities.
Such a bound on $J_{\mc{S}, s}$ usually cannot be used to derive a sharp decoupling inequality, that is, a sharp bound on $\Dec_{\mc{S}}(\delta,p)$.
Nevertheless, some features of the counting argument in Section~\ref{section_counting} remain visible in our proof of Theorem~\ref{thm:main}.
We discuss this in more detail in Remark~\ref{rem:counting-vs-dec}.

It is a bit surprising that the decoupling theory for the surface in Theorem \ref{thm:main} admits three different regimes.
This is not reflected by the lower bounds for $J_{\calS,s}$ obtained by Parsell, Prendiville, and Wooley \cite{MR3132907}, since there is no even integer in the interval $(4, 6) \subset \R$.
For this reason, we discuss lower bounds directly for decoupling inequalities in Section~\ref{sec:lower-bd}.

In Theorem~\ref{thm:classification+exponents}, we see that there are several different regimes for sharp decoupling exponents, and in case~\ref{it:complete-square} of that theorem we see that equal decoupling exponents can arise in different ways.
This can be seen as another indication that defining ``curvature'' (for instance in the spirit of the rotational curvature of Phong and Stein \cite{MR857680,MR853446}) for surfaces of co-dimension two may be a very challenging problem.
Currently, defining a curvature in this context seems still to be at the stage of considering concrete examples, see for instance Oberlin \cite{MR2078263}, Gressman \cite{MR3922044} and the references therein.

\subsection{Novelty of the proof}
To our knowledge, Theorem~\ref{thm:main} is the first instance of a sharp decoupling inequality proved despite the failure of the BCCT transversality condition \cite[Theorem 1.15 (c)]{MR2377493}.

The importance of the BCCT condition for decoupling inequalities was emphasized for instance in \cite[Conjecture 1.3]{MR3709122}, and it was verified for large classes of monomial surfaces in \cite{MR3548534,MR3994585,MR4031117}, and for certain classes of polynomial surfaces of low co-dimensions in \cite{MR3848437, MR3945730,MR4047565,MR4143735}.

In our case $(P,Q)=(r^{2},s^{2}+rt)$, the BCCT transversality condition would require
\begin{equation}
\label{eq:transversality-cond}
\dim \pi_{r,s,t} V \geq \tfrac{3}{5} \dim V
\end{equation}
for every linear subspace $V \subset \R^{5}$ and almost surely in the sense of the Lebesgue measure in $(r, s, t)\in [0, 1]^3$, where $\pi_{r,s,t}$ denotes the orthogonal projection onto the tangent space to the surface $\calS$ above the point $(r,s,t)$.
This tangent space equals
\begin{equation}
V(r,s,t) = \lin \Set{ (1,0,0,2r,t), (0,1,0,0,2s), (0,0,1,0,r) },
\end{equation}
and one sees that the condition~\eqref{eq:transversality-cond} is violated for $V = \R \times \Set{0} \times \Set{0} \times \R \times \Set{0}$.

We deal with the failure of the BCCT condition by decoupling alternatingly in the coordinates $r,t$ and in $s$.
Roughly speaking, at certain stages of the proof, we will pick the scales of $r$ and $t$ carefully, so that the surface $\mc{S}$ behaves like a curve $(s, s^2)$;  at other stages, we will pick the scale of $s$ carefully, so that $\mc{S}$ behaves like the surface $(r, t, r^2, rt)$.

How to pick these scales is crucial to the proof.
Let us briefly describe the idea here.
Let $\sigma<1$ be small but much bigger than $\delta$.
Consider a frequency region where $(r, t)$ takes values in a square of side length $\sigma$.
Without loss of generality, assume $(r, t)\in [0, \sigma]^2$.
Then, on a spatial ball of radius $\sigma^{-2}$ in $\R^5$, by the uncertainty principle, the surface $(r, s, t, r^2, s^2+rt)$ has essentially the same behavior as $(r, s, t, 0, s^2)$.
Then one can apply a ``small cap'' (which may also be referred to as ``small ball'' depending on the context) decoupling inequality for the parabola $(s, s^2)$ (Theorem~\ref{thm:small-ball-dec}) to decouple the frequency domain of $s$ into small intervals of length $\sigma^2$.

Next, we fix a frequency interval of $s$, say $[0, \sigma^2]$, and try to decouple in $r$ and $t$.
It is natural to proceed in the following way: consider a spatial ball in $\R^5$ of radius $\sigma^{-4}$.
In this ball, by the uncertainly principle, the surface $(r, s, t, r^2, s^2+rt)$ has the same behavior as $(r, 0, t, r^2, rt)$.
If we had a small cap decoupling inequality, analogous to that for the parabola, that would allow us to decouple in $r$ and $t$ into an even smaller scale, say $\sigma^{-4}$, then we would be able to iterate this argument until we reach frequency scale $\delta$, and obtain the desired decoupling inequality in $L^4$. 

Unfortunately, such a small cap decoupling inequality for the surface $(r, t, r^2, rt)$ is not available, see Example~\ref{ex:small-cap-fails}.
Instead, we will apply a decoupling inequality associated with certain spatial rectangular slabs (Theorem~\ref{thm:slab-dec}).
The use of rectangular slabs is crucial, and to our knowledge, has not previously appeared in the literature. 

\subsection*{Acknowledgment}
The authors thank Phil Gressman and Lillian Pierce for numerous insightful discussions. They also thank Ruixiang Zhang for helpful discussion. 
SG, JR, and PY thank the American Institute of Mathematics for funding two SQuaRE workshops where part of this work was done.
SG was supported in part by the NSF grant 1800274.
PY was supported in part by a General Research Fund CUHK14303817 from the Hong Kong Research Grants Council, and a direct grant for research from the Chinese University of Hong Kong (4053341).
PZ was partially supported by the Hausdorff Center for Mathematics (DFG EXC 2047). 

\section{Classification of pairs of quadratic forms in 3 variables}
In this section we prove the classification part of Theorem~\ref{thm:classification+exponents}.

\begin{proposition}\label{prop:classification}
Let $(P, Q)$ be a degenerate pair of linearly independent quadratic forms.
Moreover, assume that there is no linear change of variables in $(r, s, t)$ after which $P$ and $Q$ both omit one of the variables.
Then exactly one of the following alternatives holds.
\begin{enumerate}
\item\label{it:1} $(P, Q) \equiv (rs, rt)$,
\item\label{it:2} $(P, Q) \equiv (r^2, s^2\pm t^2)$, or
\item\label{it:3} $(P, Q) \equiv (r^2, s^2+rt)$.
\end{enumerate}
\end{proposition}

The key step of proving Proposition~\ref{prop:classification} is the following result.
\begin{lemma}\label{190616lemma2}
For two general quadratic forms in $3$ variables $P$ and $Q$, Condition \eqref{eq:det-condition} is equivalent to
\begin{equation}
\text{no non-trivial linear combination of $P,Q$ is a complete square.}
\end{equation}
\end{lemma}
\begin{proof}[Proof of Proposition~\ref{prop:classification} assuming Lemma~\ref{190616lemma2}.]
The hypothesis that $(P, Q)$ is a degenerate pair means that at least one of the conditions \eqref{eq:no-common-factor}, \eqref{eq:det-condition} fails.

Assume that \eqref{eq:no-common-factor} fails, that is, that the two quadratic forms $P$ and $Q$ share a common real linear factor.
Without loss of generality, we may assume that the common factor is $r$.
Then, up to a linear change of variables, there are two cases, $(P, Q)=(r^2, rs)$ or $(P, Q)=(rs, rt)$.
Here we used the assumption that $P$ and $Q$ are linearly independent.
The former case is not admissible, as the $t$ variable is omitted.

Suppose now that \eqref{eq:no-common-factor} holds and \eqref{eq:det-condition} fails.
By Lemma~\ref{190616lemma2}, \eqref{eq:det-condition} fails if and only if some non-trivial linear combination of $P,Q$ is a complete square.
Hence, after a change of variables as in \eqref{L1L2}, we may assume $P(r, s, t)=r^2$.

Now consider $Q(0, s, t)$, which is a quadratic form of two variables. First of all, we know that it cannot have rank zero, as otherwise $Q(r, s, t)$ will share a common factor with $P(r, s, t)=r^2$.
Therefore, $Q(0, s, t)$ can have rank either one or two.
Making a change of variables in $s$ and $t$, we may assume that $Q(0, s, t)$ equals either $s^{2} \pm t^{2}$ (if it has rank $2$) or $s^{2}$ (if it has rank $1$).
In the rank $2$ case, we have
\begin{equation}\label{rank2}
Q(r,s,t) =
s^2\pm t^2+ c_1 r^2+ c_2 r s+c_3 rt.
\end{equation}
Here $c_1, c_2, c_3$ are real numbers. We now add multiples of $P(r, s, t)=r^2$ to $Q(r, s, t)$ and complete squares. This process removes all the mixed terms in \eqref{rank2} and hence $(P, Q)\equiv (r^2, s^2\pm t^2)$.

The case of $Q(0, s, t)$ having rank one is similar, but one of the mixed terms cannot be removed, hence $(P,Q) \equiv (r^{2},s^{2}+crt)$.
The coefficient $c$ does not vanish, since otherwise $(P,Q)$ would omit the variable $t$.
Rescaling the variable $t$, we may assume $c=1$.
\end{proof}
\begin{proof}[Proof of Lemma~\ref{190616lemma2}.]
If some non-trivial linear combination of $P,Q$ is of the form $(u_{1}r+u_{2}s+u_{3}t)^{2}$, then Condition \eqref{eq:det-condition} fails for that $u$.

Conversely, suppose that Condition \eqref{eq:det-condition} fails for some $u\neq 0$.
Without loss of generality, we may assume $u=(0,0,1)$.
Write $P=a_{rr}r^{2}+a_{ss}s^{2}+a_{tt}t^{2}+a_{rs}rs+a_{rt}rt+a_{st}st$ and $Q=b_{rr}r^{2}+b_{ss}s^{2}+b_{tt}t^{2}+b_{rs}rs+b_{rt}rt+b_{st}st$.
Then
\begin{equation}
\begin{split}
0 &\equiv
\det(\nabla P, \nabla Q, u)
=
\det \begin{pmatrix}
\partial_{r} P & \partial_{r} Q \\
\partial_{s} P & \partial_{s} Q
\end{pmatrix}
\\ &=
\det \begin{pmatrix}
2ra_{rr}+sa_{rs}+ta_{rt} & 2rb_{rr} + sb_{rs} + tb_{rt} \\
2sa_{ss}+ra_{rs}+ta_{st} & 2sb_{ss} + rb_{rs} + tb_{st}
\end{pmatrix}
\\ &=
2r^{2} (a_{rr}b_{rs}-a_{rs}b_{rr}) + 2s^{2} (a_{rs}b_{ss}-a_{ss}b_{rs}) + t^{2} (a_{rt}b_{st}-a_{st}b_{rt})
\\ &+
rs (4a_{rr}b_{ss}+a_{rs}b_{rs}-4a_{ss}b_{rr}-a_{rs}b_{rs})
\\ &+
rt (2a_{rr}b_{st}+a_{rt}b_{rs}-a_{rs}b_{rt}-2a_{st}b_{rr})
\\ &+
st (a_{rs}b_{st}+2a_{rt}b_{ss}-2a_{ss}b_{rt}-a_{st}b_{rs}).
\end{split}
\end{equation}
Since all coefficients must vanish, we obtain
\begin{multline}
\label{eq:ab}
0 = a_{rr}b_{rs}-a_{rs}b_{rr}
= a_{rs}b_{ss}-a_{ss}b_{rs}
= a_{rt}b_{st}-a_{st}b_{rt}
= a_{rr}b_{ss}-a_{ss}b_{rr}
\\ = 2a_{rr}b_{st}+a_{rt}b_{rs}-a_{rs}b_{rt}-2a_{st}b_{rr}
= a_{rs}b_{st}+2a_{rt}b_{ss}-2a_{ss}b_{rt}-a_{st}b_{rs}.
\end{multline}
Replacing $(P,Q)$ by suitable linear combinations, we may assume without loss of generality $a_{rs}=0$.
We distinguish several cases.

Case 1: $b_{rs} = 0$.
Then the equations simplify to
\begin{equation}
0 = a_{rt}b_{st}-a_{st}b_{rt}
= a_{rr}b_{ss}-a_{ss}b_{rr}
= a_{rr}b_{st}-a_{st}b_{rr}
= a_{rt}b_{ss}-a_{ss}b_{rt}.
\end{equation}
This shows that $(a_{rr},a_{ss},a_{rt},a_{st})$ and $(b_{rr},b_{ss},b_{rt},b_{st})$ lie in the same one-dimensional subspace of $\R^{4}$.
Hence, subtracting a suitable multiple of $Q$ from $P$, we may assume $(a_{rr},a_{ss},a_{rt},a_{st})=0$.
But then $P=a_{tt}t^{2}$, and we are done.

Case 2: $b_{rs} \neq 0$.
Then from the first two equations in \eqref{eq:ab} we obtain $a_{rr}=a_{ss}=0$, and the remaining equations simplify to
\begin{equation}
0 = a_{rt}b_{st}-a_{st}b_{rt}
= a_{rt}b_{rs}-2a_{st}b_{rr}
= 2a_{rt}b_{ss}-a_{st}b_{rs}.
\end{equation}
Case 2.1: if $a_{rt}=a_{st}=0$, then $P=a_{tt}t^{2}$, and we are done.

\noindent Case 2.2: if exactly one of $a_{rt},a_{st}$ vanishes, then suppose without loss of generality $a_{rt}=0$ and $a_{st}\neq 0$.
Then from the last equation $b_{rs}=0$, contradiction.

\noindent Case 2.3: both $a_{rt}\neq 0$ and $a_{st}\neq 0$.
Then, multiplying the last two equations, we obtain
\begin{equation}
2a_{st}b_{rr} \cdot 2a_{rt}b_{ss} = a_{rt}b_{rs} \cdot a_{st}b_{rs}.
\end{equation}
Since all $a$'s don't vanish, this gives $4 b_{rr} b_{ss} = b_{rs}^{2}$.
Hence $b_{rr}r^{2}+b_{ss}s^{2}+b_{rs}rs$ is a complete square.
Making a change of variables only in $(r,s)$, we may assume $b_{rs}=0$.
Notice that we retain the relation $a_{rs}=0$ after this change of variables, since we have $a_{rr}=a_{ss}=0$ in Case 2.
Hence we are back in Case 1.
\end{proof}

\section{Sharp decouplings for the surface $(r, s, t, rs, rt)$}
In this section, we will prove the upper bound in \eqref{rsrt_decexp}.
\begin{proposition}
\label{190707prop3}
Let $\mc{S}$ be the surface given by $(r, s, t, rs, rt)$.
Then, for every $\epsilon>0$, we have
\begin{equation}
\Dec_{\mc{S}}(\delta, p) \lesim_{\epsilon, p}
\begin{cases}
\delta^{-(2-\frac 4 p)-\epsilon} & \text{ if } 2\le p\le 6,\\
\delta^{-(3-\frac{10}p)-\epsilon} & \text{ if } 6\le p\le \infty.
\end{cases}
\end{equation}
\end{proposition}

\begin{proof}
By interpolation with orthogonality at $p=2$ and a trivial estimate at $p=\infty$ (see \eqref{eq:dec-interpolation}), it suffices to prove the case $p=6$.
First, we notice that
\begin{equation}
(rs, rt)\equiv (rs, r^2+rt).
\end{equation}
Denote
\begin{equation}
\mc{S}'=\{(r, s, t, rs, r^2+rt): (r, s, t)\in [0, 1]^3\}.
\end{equation}
Our goal is to prove that
\begin{equation}
\norm[\big]{ \sum_{\Box\in \Part{\delta}} f_{\Box} }_6
\lesim_{\epsilon}
\delta^{-4(\frac 1 2-\frac 1 6)-\epsilon}
\Big(\sum_{\Box} \norm{f_{\Box}}_6^6 \Big)^{1/6},
\end{equation}
where
\begin{equation}
\supp(\widehat{f_{\Box}})\subset
\Set{ (r, s, t, rs+\delta', r^2+rt+\delta'') \given (r, s, t)\in \Box, \abs{\delta'},\abs{\delta''}\le \delta^2 }.
\end{equation}
For an integer $0\le j\le \delta^{-1}$, let
\begin{equation}
S_j=[0, 1]\times [0, 1]\times [j\delta, (j+1)\delta].
\end{equation}
By flat decoupling \eqref{eq:flat-dec}, we obtain
\begin{equation}
\norm{\sum_{\Box\in \Part{\delta}} f_{\Box}}_6 \lesim \delta^{-2(\frac 1 2-\frac 1 6)} \Big(\sum_{j} \norm{ \sum_{\Box\subset S_j} f_{\Box}}_6^6 \Big)^{1/6}.
\end{equation}
It remains to prove
\begin{equation}\label{eq:L6decS''}
\norm{ \sum_{\Box\subset S_j} f_{\Box}}_6 \lesim_{\epsilon} \delta^{-2(\frac 1 2-\frac 1 6)-\epsilon} \Big(\sum_{\Box} \norm{f_{\Box}}_6^6 \Big)^{1/6},
\end{equation}
uniformly in $j$.
To this end, we use the decoupling inequality for the surface
\begin{equation}
\label{eq:calS''}
\mc{S}'':=\{(r, s, rs, r^2): (r, s)\in [0, 1]^2\},
\end{equation}
which was proved in Bourgain and Demeter \cite{BD16} (see also \cite{GZo19}, where this result is discussed from a more general perspective).
In our notation, \cite[Theorem 1.2]{BD16} implies that, for every $\epsilon>0$, we have
\begin{equation}
\label{eq:L6decBD16}
\Dec_{\mc{S}''}(\delta, 6)
\lesim_{\epsilon}
\delta^{-2(\frac 1 2-\frac 1 6)-\epsilon}.
\end{equation}
Using \cite[Theorem 2.2]{GZo19} with $\calH = \Set{(r,s,t) \given t=0}$, we can now deduce \eqref{eq:L6decS''} with $j=0$.
The estimates \eqref{eq:L6decS''} for other values of $j$ can be obtained from the case $j=0$ using the affine transformation
\[
(r,s,t,\xi,\eta) \mapsto (r,s,t+t_{0},\xi,\eta+rt_{0})
\]
in the frequency space, where $t_{0}=j\delta$.
\end{proof}

\section{A counting argument}\label{section_counting}

Consider the Diophantine system \eqref{diophantine} with $P(x, y, z)=x^2$ and $Q(x, y, z)=y^2+xz$. In this section, we will give a direct proof of the estimate 
\begin{equation}
\label{eq:J-with-s=2}
J_{\mc{S}, 2}(N)\lesim_{\epsilon} N^{6+\epsilon},
\end{equation}
for every $\epsilon>0$, without invoking decoupling theory.
Recall that this corresponds to the case $p=4$ in Theorem~\ref{thm:main}, which is the most interesting case there.
As mentioned in the introduction, the argument used in the following proof will shed some light on how to prove the related decoupling inequality in Theorem~\ref{thm:main}.

In the current situation, the system of equations \eqref{diophantine} becomes
\begin{equation}\label{diophantine_explicit}
\begin{split}
 x_1+x_2 &=x_3+x_4,\\
 y_1+y_2 &=y_3+y_4,\\
 z_1+z_2 &=z_3+z_4,\\
 x_1^2+x_2^2&=x_3^2+x_4^2,\\
 y_1^2 + x_1 z_1 + y_2^2 + x_2 z_2 &= y_3^2 + x_3 z_3 + y_4^2 + x_4 z_4.
\end{split}
\end{equation}
The first and fourth equations in \eqref{diophantine_explicit} imply that $\{x_1,x_2\}$ is a permutation of $\{x_3,x_4\}$. Without loss of generality let us assume that $x_1=x_3$ and $x_2=x_4$.
Also keeping in mind that $z_1-z_3=z_4-z_2$, the last equation in
\eqref{diophantine_explicit} can then be written as
\begin{equation} y_1^2  - y_3^2 + (x_1-x_2)(z_1 - z_3) = y_4^2 - y_2^2.\end{equation}
We now distinguish two cases: $x_1=x_2$ and $x_1\not=x_2$.

Case 1: $x_1=x_2$. Then we have $x_1 =x_2=x_3=x_4$ and this is the only constraint on the $x_i$ variables. Similarly, the only constraint on the $y_i$ variables now becomes
\begin{equation}
\begin{split}
y_1 + y_2 &= y_3 + y_4,\\
y_1^2  - y_3^2 &= y_4^2 - y_2^2.
\end{split}
\end{equation}
Finally, the only constraint on the $z_i$ variables is the linear equation $z_1+z_2=z_3+z_4$. To summarize, this case leads to a contribution of $N\cdot N^2\cdot N^3=N^6$ to  $J_{\mc{S}, 2}(N)$.

Case 2: $x_1 \neq x_2$. In this case we have two free choices among the $x_i$ variables. Suppose that $x_1,x_2,x_3,x_4$ have been fixed.

Case 2.1: $z_1=z_3$. Then also $z_2 = z_4$, so we may choose two of the $z_i$ variables freely. Suppose that $z_1,z_2,z_3,z_4$ have been fixed.
The remaining constraints are now
\begin{equation}
\begin{split}
y_1 + y_2 &= y_3 + y_4,\\
y_1^2  - y_3^2 &= y_4^2 - y_2^2,
\end{split}
\end{equation}
which gives two free choices of $y_i$ variables. Summarizing, this case yields a contribution of $\approx N^6$ to $J_{\mc{S}, 2}(N)$.
 
Case 2.2: $z_1\not=z_3$. This is the critical case. First note that there are $\approx N^3$ valid choices of the $z_i$ variables. Assume that $z_1,z_2,z_3,z_4$ have been fixed. 
It remains to analyze the constraints on the remaining variables $y_1,y_2,y_3,y_4$, which can be written as
\begin{equation}\label{eqn:solcnt-crit}
\begin{split}
y_1 -y_3 &= y_4 -y_2,\\
y_1^2  - y_3^2 + C &= y_4^2 - y_2^2,
\end{split}
\end{equation}
where $C=(x_1-x_2)(z_1 - z_3) \neq 0$. We will now make critical use of the fact that all involved quantities are integers. Observe that necessarily $y_1\not=y_3$. Next, the first equation implies that $y_4^2-y_2^2$ is divisible by $y_1-y_3$. Since also $y_1^2-y_3^2$ is divisible by $y_1-y_3$, the second equation implies that $y_1-y_3$ must divide $C$.
Since $C$ is $\lesssim N^2$, we have $d(C)\lesssim_\epsilon N^\epsilon$ for all $\epsilon>0$, where $d(C)$ denotes the number of divisors of $C$. Let $D$ be one of these divisors and suppose that $y_1-y_3=D$. We then have the constraints
\begin{equation}
\begin{split}
y_1 -y_3 &= D,\\
y_4 -y_2 &= D,\\
y_1+y_3 + \tfrac{C}{D} &= y_4  + y_2.
\end{split}
\end{equation}
For each fixed $D$, there are $\lesssim N$ valid choices of $y_1,y_2,y_3,y_4$. Summarizing, this case gives a contribution of $\lesim_\epsilon N^{6+\epsilon}$ to $J_{\mc{S}, 2}(N)$, for all $\epsilon>0$.

\begin{remark}
Using the average bound $\sum_{C \leq N} d(C) \lesssim N \log N$, see e.g.\ \cite[Theorem 2.3]{MR2378655}, instead of a pointwise bound on $d(C)$, the bound \eqref{eq:J-with-s=2} can be further improved to $J_{\mc{S}, 2}(N) \lesssim N^{6} (\log N)^{2}$.
\end{remark}

\begin{remark}
\label{rem:counting-vs-dec}
Let us now mention the analogies between the above argument and the proof of the case $p=4$ of Theorem~\ref{thm:main} below.
The distinction between the $x$'s and $z$'s from the $y$'s in our solution counting argument above motivates us to decouple alternately in the $(r,t)$ variables and the $s$ variable (i.e.\ the alternate use of Propositions~\ref{190618prop5} and \ref{190618prop6}) in Section~\ref{section4aa}.
The fact that solutions of \eqref{eqn:solcnt-crit} are counted for fixed $x$'s and $z$'s corresponds to the fiberwise estimate \eqref{eq:21}.
Indeed, solutions are counted on the Fourier side of the decoupling inequality, so fixing the second spatial variable in \eqref{eq:21} corresponds to considering all $y_{1},\dotsc,y_{4}$ simultaneously in \eqref{diophantine_explicit}.
The use of the quadratic system \eqref{eqn:solcnt-crit} for $y_{1},\dotsc,y_{4}$ corresponds to the use of the small ball decoupling inequality of the parabola (Theorem~\ref{thm:small-ball-dec}) in Proposition~\ref{190618prop5}.
The initial distinction between the cases $x_{1}=x_{2}$ and $x_{1}\neq x_{2}$ corresponds to the broad/narrow dichotomy in the Bourgain--Guth argument in the proof of Proposition~\ref{190707prop8} below, which will be run only in the $r$ variable.
\end{remark}

\section{Sharp decouplings for the surface $(r, s, t, r^2, s^2+rt)$}\label{section4aa}
In this section, we will prove Theorem~\ref{thm:main}.
By interpolation (see \eqref{eq:dec-interpolation}) with orthogonality at $p=2$ and a trivial estimate at $p=\infty$, it suffices to prove the cases $p=4$ and $p=6$.
For $p=6$ we can use the same argument as in Proposition~\ref{190707prop3}, using flat decoupling in the $t$ variable and the decoupling inequality of \cite{BD16} for the surface $(r,s,r^2, s^2)$, lifted to 5 dimensions using \cite[Theorem 2.2]{GZo19}.

It remains to consider $p=4$.
The strategy of the proof for this case is already sketched in the introduction and motivated in Remark~\ref{rem:counting-vs-dec}.
Therefore we will enter the proof directly. 

\begin{notation}
For a dyadic number $\delta \in (0,1)$ and a dyadic box $\alpha\subset [0, 1]^3$ with side lengths at least $\delta$, we use $\Part[\alpha]{\delta}$ to denote the partition of $\alpha$ into dyadic cubes of side length $\delta$.
For three real numbers $k_1, k_2, k_3$, we use $\mc{P}^{(k_1, k_2, k_3)}(\alpha, \sigma)$ to denote a partition of $\alpha$ into rectangular boxes of dimension $\sigma^{k_1}\times \sigma^{k_2}\times \sigma^{k_3}$. We also write $\mathcal{P}(\delta)=\Part[\unitcube]{\delta}$ and $\mathcal{P}^{(k_1,k_2,k_3)}(\sigma)=\mathcal{P}^{(k_1,k_2,k_3)}([0,1]^3, \sigma)$ for brevity.
\end{notation}

Theorem~\ref{thm:main} will be proved by iterating the following two propositions, which decouple in different coordinates.

\begin{proposition}\label{190618prop5}
Let $\alpha_0 \in \mc{P}^{(1,0,1)}(\sigma)$.
For each $\alpha\in \mc{P}^{(1, 2, 1)}(\alpha_0, \sigma)$, let $g_{\alpha}$ be a function with
\begin{equation}
\label{eq:10}
\supp(\widehat{g_{\alpha}})
\subseteq
\Set{ (r, s, t, r^2+\sigma', s^2+rt+\sigma'') \given (r, s, t)\in \alpha, \abs{\sigma'},\abs{\sigma''}\le \sigma^2 }.
\end{equation}
Then, for every $\epsilon'>0$, we have
\begin{equation}
\label{eq:4}
\norm[\Big]{\sum_{\alpha\in \mc{P}^{(1, 2, 1)}(\alpha_0, \sigma)} g_{\alpha}}_4
\lesim_{\epsilon'}
\sigma^{-2(\frac 1 2-\frac 1 4)-\epsilon'}
\Big(\sum_{\alpha\in \mc{P}^{(1, 2, 1)}(\alpha_0, \sigma)} \norm{g_{\alpha} }_4^4 \Big)^{1/4},
\end{equation}
uniformly in $\alpha_{0}$ and $g_{\alpha}$.
\end{proposition}

\begin{proposition}\label{190618prop6}
Let $\alpha_0 \in \mc{P}^{(0,1,0)}(\sigma)$.
For each $\alpha\in \mc{P}^{(2, 1, 2)}(\alpha_0, \sigma)$, let $g_{\alpha}$ be a function with
\begin{equation}
\label{eq:11}
\supp(\widehat{g_{\alpha}})
\subseteq
\Set{ (r, s, t, r^2+\sigma', s^2+rt+\sigma'') \given (r, s, t)\in \alpha, \abs{\sigma'}\le \sigma^4, \abs{\sigma''}\le \sigma^2 }.
\end{equation}
Then, for every $\epsilon'>0$, we have
\begin{equation}\label{small_cap_rectangle}
\norm[\Big]{\sum_{\alpha\in \mc{P}^{(2, 1, 2)}(\alpha_0, \sigma)} g_{\alpha}}_4
\lesim_{\epsilon'}
\sigma^{-4(\frac 1 2-\frac 1 4)-\epsilon'}
\Big(\sum_{\alpha\in \mc{P}^{(2, 1, 2)}(\alpha_0, \sigma)} \norm{g_{\alpha} }_4^4 \Big)^{1/4},
\end{equation}
uniformly in $\alpha_{0}$ and $g_{\alpha}$.
\end{proposition}

The Fourier support restrictions in \eqref{eq:10} and \eqref{eq:11} arise naturally in the proofs (see Remarks~\ref{rmk5.1}, \ref{rmk5.2} and \ref{rmk5.3} below),
but it would be sufficient to prove the above results under the more restrictive conditions $\abs{\sigma'},\abs{\sigma''} \leq \sigma^{4}$. 

The proofs of Propositions~\ref{190618prop5} and~\ref{190618prop6}, as well as Theorem~\ref{thm:main}, rely on translation-dilation invariance, which we now explain.
Let
\[
\alpha_{0} = [0,\sigma^{k_{1}}] \times [0,\sigma^{k_{2}}] \times [0,\sigma^{k_{3}}],
\quad
\alpha = (r_0,s_0,t_0) + \tau \alpha_{0}.
\]
Then, the affine map
\begin{multline}
\label{eq:affine-scaling}
L_{\alpha}(\xi) := (\tau\xi_1 + r_0, \tau\xi_2 + s_0, \tau\xi_3 + t_0,\\
\tau^{2}\xi_4 + \tau 2r_0 \xi_1,
\tau^{2}\xi_5 + \tau(2s_0 \xi_2 + t_0 \xi_1 + r_0 \xi_3) + s_0^2 + r_0t_0 )
\end{multline}
maps frequency parallelepipeds such as in \eqref{eq:fourier-supp}, \eqref{eq:10}, and \eqref{eq:11} to other frequency parallelepipeds of a similar form.
Thus, in order to decouple on $\alpha$, we will first pull the Fourier transforms of the relevant functions back by $L_{\alpha}$, and decouple on $[0,\sigma^{k_{1}}] \times [0,\sigma^{k_{2}}] \times [0,\sigma^{k_{3}}]$ instead, where
\[
(k_{1},k_{2},k_{3})
\in \Set{ (0,0,0), (0,1,0), (1,0,1) }.
\]

\begin{proof}[Proof of Theorem~\ref{thm:main} with $p=4$ assuming Propositions~\ref{190618prop5} and~\ref{190618prop6}]
For a dyadic box $\alpha \subseteq [0,1]^{3}$, we write
\begin{equation}
f_{\alpha} := \sum_{\Box\in\Part[\alpha]{\delta}} f_{\Box}.
\end{equation}
Set $\sigma=\delta^{\epsilon}$.
By flat decoupling \eqref{eq:flat-dec}, we obtain
\begin{equation}
\label{eq:9}
\norm[\Big]{\sum_{\alpha_0\in \mc{P}^{(1, 0, 1)}(\sigma)} f_{\alpha_{0}} }_4
\lesim
\sigma^{-2(1-\frac 2 4)} \Big( \sum_{\alpha_0\in \mc{P}^{(1, 0, 1)}(\sigma)}\norm[\big]{ f_{\alpha_{0}} }^4_4 \Big)^{1/4}.
\end{equation}
We iterate the following two estimates.
Let $k \in \Set{0,1,\dotsc}$ with $\delta \leq \sigma^{2k+3}$.

Given $\alpha \in \mc{P}^{(2k+1, 2k, 2k+1)}(\sigma)$, by a rescaled version of Proposition~\ref{190618prop5}, we obtain
\begin{equation}
\label{eq:7}
\begin{split}
\norm[\big]{ f_{\alpha} }_4
&=
\norm[\Big]{ \sum_{\alpha' \in \mc{P}^{(2k+1, 2k+2, 2k+1)}(\alpha, \sigma)} f_{\alpha'} }_4
\\ & \lesim_{\epsilon}
\sigma^{-2(\frac 1 2-\frac 1 4)-\epsilon}
\Big( \sum_{\alpha'\in \mc{P}^{(2k+1, 2k+2, 2k+1)}(\alpha, \sigma)} \norm[\big]{ f_{\alpha'} }_4^4\Big)^{1/4}.
\end{split}
\end{equation}

Given $\alpha\in \mc{P}^{(2k+1, 2k+2, 2k+1)}(\sigma)$, by a rescaled version of Proposition~\ref{190618prop6}, we obtain
\begin{equation}
\label{eq:8}
\begin{split}
\norm[\big]{ f_{\alpha} }_4
&=
\norm[\Big]{ \sum_{\alpha'\in \mc{P}^{(2k+3, 2k+2, 2k+3)}(\alpha, \sigma)} f_{\alpha'} }_4\\
& \lesim_{\epsilon}
\sigma^{-4(\frac 1 2-\frac 1 4)-\epsilon}
\Big(\sum_{\alpha'\in \mc{P}^{(2k+3, 2k+2, 2k+3)}(\alpha, \sigma)} \norm[\big]{ f_{\alpha'} }_4^4\Big)^{1/4}.
\end{split}
\end{equation}
Let $K$ be the largest integer such that $\delta \leq \sigma^{2K+1}$.
Using \eqref{eq:9} and applying the estimates \eqref{eq:7} and \eqref{eq:8}  for $k=0,\dots,K-1$, we obtain
\begin{equation}
\norm[\Big]{ \sum_{\Box\in \Part{\delta}} f_{\Box} }_4
\lesssim_{K,\epsilon}
\sigma^{-(6K+4)(\frac12-\frac14)-2K\epsilon}
\Big(\sum_{\alpha\in \mc{P}^{(2K+1, 2K, 2K+1)}(\sigma)} \norm[\big]{ f_{\alpha} }_4^4\Big)^{1/4}.
\end{equation}

For every $\alpha\in \mc{P}^{(2K+1, 2K, 2K+1)}(\sigma)$, we have $\abs{\Part[\alpha]{\delta}} \leq \sigma^{-7}$.
Hence, by flat decoupling \eqref{eq:flat-dec}, we obtain
\begin{equation}
\norm[\big]{ f_{\alpha} }_4
\lesssim
\sigma^{-7 \cdot (1-\frac24)}
\Big(\sum_{\Box \in \Part[\alpha]{\delta}} \norm[\big]{ f_{\Box} }_4^4\Big)^{1/4}.
\end{equation}
Combining the last two estimates, we obtain
\begin{equation}
\norm[\Big]{ \sum_{\Box\in \Part{\delta}} f_{\Box} }_4
\lesssim_{K,\epsilon}
\sigma^{-(6K+18)(\frac12-\frac14)-2K\epsilon}
\Big(\sum_{\Box \in \Part{\delta}} \norm[\big]{ f_{\Box} }_4^4\Big)^{1/4}.
\end{equation}
Since $\sigma^{-2K} \leq \delta^{-1}$, this concludes the proof.
\end{proof}

It remains to prove Propositions \ref{190618prop5} and \ref{190618prop6}, which will be our objective in the next two subsections. The key ingredients are a decoupling inequality on  ``small balls'' for the parabola $\Set{(s,s^2) \given s \in [0,1]}$ and a decoupling inequality on ``thin slabs'' for the surface $\Set{(r,t,r^2,rt) \given (r,t) \in [0,1]^2}$. The smallness and thinness of these balls and slabs are what allowed us to decouple certain frequency variables down to scale $\sigma^2$ in Propositions \ref{190618prop5} and \ref{190618prop6} when the other frequency variables are limited to an interval of length $\sigma$. This is crucial in letting us make progress, as we decouple in alternate coordinates in the above proof of Theorem~\ref{thm:main}.

\subsection{Decoupling on small balls and proof of Proposition~\ref{190618prop5}}\label{subsection4.1}
We will need the following ``small ball'' decoupling inequality. 
The term ``small ball'' refers to the fact that it can be localized to spatial scale $\delta^{-1}$, whereas the usual decoupling inequality \eqref{eq:dec-const} can only be localized to the larger spatial scale $\delta^{-2}$.
As a side note, optimal decoupling inequalities for the parabola at spatial scales between $\delta^{-1}$ and $\delta^{-2}$ were recently established in \cite{MR4153908}.
In that paper, ``small ball'' decoupling inequalities are referred to as ``small cap'' inequalities.
They mean the same thing: One features the spatial side of the problem, while the other features the frequency side.
Here, we prefer the name ``small ball'', because in Proposition~\ref{190618prop6} we need a decoupling inequality similar in spirit that also features the spatial side of the problem. 

\begin{theorem}[{cf. \cite[Lemma 4.2]{GLY}}]
\label{thm:small-ball-dec}
Let $\delta \in (0,1)$ be a dyadic number, and for each $\theta \in \Part[\unitint]{\delta}$ let $f_{\theta}$ be a tempered distribution on $\R^{2}$ with
\begin{equation}
\label{eq:small-ball-dec:fourier-supp}
\supp \widehat{f_{\theta}}
\subseteq
\Set{ (s,s^{2}+\delta') \given s \in \theta, \abs{\delta'} \leq \delta }.
\end{equation}
Then, for every $\epsilon>0$, we have
\begin{equation}
\norm[\Big]{ \sum_{\theta\in\Part[\unitint]{\delta}} f_{\theta}}_{L^{4}(\R^{2})}
\lesssim_{\epsilon}
\delta^{-(\frac12-\frac14)-\epsilon}
\Bigl( \sum_{\theta\in\Part[\unitint]{\delta}} \norm[\big]{ f_{\theta}}_{L^{4}(\R^{2})}^{4} \Bigr)^{1/4}.
\end{equation}
\end{theorem}

In \cite[Lemma 4.2]{GLY}, a version of Theorem~\ref{thm:small-ball-dec} is stated for the extension operator.
Although it is possible to deduce Theorem~\ref{thm:small-ball-dec} from \cite[Lemma 4.2]{GLY} by the argument in \cite[Section 5]{BD-guide}, it is preferable to observe that the proof continues to work at the level of generality in Theorem~\ref{thm:small-ball-dec}.

\begin{proof}[Proof of Proposition~\ref{190618prop5}]
By affine scaling, we may assume that $\alpha_{0} = [0,\sigma] \times [0,1] \times [0,\sigma]$.
The inequality \eqref{eq:4} will follow from the fiberwise inequality
\begin{equation}
\label{eq:6}
\norm[\Big]{\sum_{\alpha\in \mc{P}^{(1, 2, 1)}(\alpha_0, \sigma)} g_{\alpha}(x_{1},\cdot,x_{3},x_{4},\cdot)}_{L^{4}(\R^{2})}
\lesim_{\epsilon'}
\sigma^{-2(\frac 1 2-\frac 1 4)-\epsilon'} \Big(\sum_{\alpha}\norm{g_{\alpha}(x_{1},\cdot,x_{3},x_{4},\cdot) }_{L^{4}(\R^{2})}^4 \Big)^{1/4}
\end{equation}
with a constant independent of $x_{1},x_{3},x_{4}$.
The inequality \eqref{eq:6} holds by Theorem~\ref{thm:small-ball-dec}, since for every choice of $x_{1},x_{3},x_{4}$ the Fourier support of
\begin{equation}
g_{\alpha}(x_{1},\cdot,x_{3},x_{4},\cdot)
\end{equation}
is contained in an $O(\sigma^{2})$-neighborhood of a $\sigma^{2}$-arc of the unit parabola in $\R^{2}$. Indeed, the Fourier support is contained in the projection of the right hand side of \eqref{eq:10} to $\R^{2}$ by omitting the first, third and fourth coordinate. Since $|rt| \leq \sigma^2$ when $(r,s,t) \in \alpha_0$, the projection is contained inside  $\Set{ (s, s^2+\sigma''') \given \abs{\sigma'''}\le 2\sigma^2 }$, as claimed. 
\end{proof}

\begin{remark} \label{rmk5.1}
The above calculation of the projection shows that even if we had the stronger condition $|\sigma''| \leq \sigma^4$ on the right hand side of \eqref{eq:10}, the proof of Proposition~\ref{190618prop5} will not become easier: the projection of the right hand side of \eqref{eq:10} will still only be in an $O(\sigma^{2})$-neighborhood of a $\sigma^{2}$-arc of the unit parabola in $\R^{2}$ (thanks to the contribution from the term $rt$). We would still need to use the small ball decoupling inequality in Theorem \ref{thm:small-ball-dec}, since the goal was to decouple in the $s$ coordinate down to scale smaller than $\sigma$.
\end{remark}

\subsection{Decoupling on thin slabs and proof of Proposition~\ref{190618prop6}} \label{subsect:slab}
In view of the proof of Proposition~\ref{190618prop5}, it would be natural to use a small ball decoupling for the $2$-dimensional surface $(r, t, r^2, rt)$ in $\R^{4}$. If we had such a small ball decoupling, then we could hope to have the estimate \eqref{small_cap_rectangle} under the assumption that 
\begin{equation}
\label{eq:1111}
\supp(\widehat{g_{\alpha}})
\subseteq
\Set{ (r, s, t, r^2+\sigma', s^2+rt+\sigma'') \given (r, s, t)\in \alpha, \abs{\sigma'}\le \sigma^2, \abs{\sigma''}\le \sigma^2 }.
\end{equation} 
One would then be able to finish the proof of Theorem \ref{thm:main} using the same bootstrapping argument. Unfortunately, as the following example shows, although such a small cap decoupling holds for the more ``elliptic'' surface $(r, t, r^2, t^2)$, it fails for its ``hyperbolic'' variant $(r, t, r^2, rt)$.
\begin{example}
\label{ex:small-cap-fails}
For each $\theta\in \Part[{[0,\delta^{1/2}]\times[0,1]}]{\delta}$, let $f_{\theta}$ be such that $\widehat{f_\theta}$ is a non-negative smooth bump function with $\int \widehat{f_{\theta}}=1$ supported in and adapted to a cube of sidelength $\approx \delta$ contained in
\begin{equation}
\Set{ (r,t,r^2+\delta',rt+\delta'')\given (r,t)\in\theta, \abs{\delta'}\le \delta, \abs{\delta''}\le \delta }.
\end{equation}
For the remaining $\theta\in \Part[\unitsquare]{\delta}\setminus \Part[{[0,\delta^{1/2}]\times[0,1]}]{\delta}$, set $f_\theta=0$.
Then $\cup_{\theta} \supp \widehat{f_{\theta}}$ is contained in a box of size $\approx \delta^{1/2} \times 1 \times \delta \times \delta^{1/2}$.
Hence, $|\sum_{\theta\in\Part[\unitsquare]{\delta}} f_{\theta}|\gtrsim \delta^{-3/2}$ on a slab of dimensions $\approx \delta^{-1/2}\times 1\times \delta^{-1}\times \delta^{-1/2}$ centered at the origin, and it follows that
\begin{equation}
\norm[\Big]{ \sum_{\theta\in\Part[\unitsquare]{\delta}} f_{\theta}}_{L^{4}(\R^{4})}\gtrsim \delta^{-2}.
\end{equation}
On the other hand,
\begin{equation}
\delta^{-2(\frac12-\frac14)} \Bigl( \sum_{\theta\in\Part[\unitsquare]{\delta}} \norm[\big]{ f_{\theta}}_{L^{4}(\R^{4})}^{4} \Bigr)^{1/4}
\approx
\delta^{-2(\frac12-\frac14)} \Bigl( \delta^{-3/2} \cdot \delta^{-4} \Bigr)^{1/4}
=
\delta^{-(2-\frac18)},
\end{equation}
much smaller than $\delta^{-2}$.
\end{example}

The above example shows that we do not have a small ball decoupling for the surface $(r, t, r^2, rt)$ in $\R^{4}$. Instead, we will prove a slightly weaker result, which we call ``decoupling on thin slabs'', since it can be localized to thin slabs of size $\delta^{-2}\times\delta^{-2}\times\delta^{-2}\times\delta^{-1}$.

For a dyadic rectangle $R\subset [0, 1]^2$, we let $\Part[R]{\delta}$ be the partition of $R$ into squares of side length $\delta$.
We denote by $\mc{V}(\delta)$ the smallest constant such that, for every collection of functions $g_{\alpha}$ indexed by $\alpha \in \Part[\unitsquare]{\delta}$ with
\begin{equation}\label{eq:2}
\supp(\widehat{g_{\alpha}})
\subset
\Set{ (r,t,r^2+\delta',rt+\delta'') \given (r,t) \in \alpha,\, \abs{\delta'} \leq \delta^2,\, \abs{\delta''} \leq \delta },
\end{equation}
the following inequality holds:
\begin{equation}
\norm[\Big]{ \sum_{\alpha \in \Part[\unitsquare]{\delta}} g_{\alpha} }_{L^4(\R^4)}
\leq
\mc{V}(\delta) \biggl( \sum_{\alpha \in \Part[\unitsquare]{\delta}} \norm{g_{\alpha}}_{L^4(\R^4)}^4 \biggr)^{\frac14}.
\end{equation}
\begin{theorem}
\label{thm:slab-dec}
For every $\epsilon>0$ and every dyadic $\delta \in (0,1)$, we have
\begin{equation}
\mc{V}(\delta) \lesssim_{\epsilon} \delta^{-1/2-\epsilon}.
\end{equation}
\end{theorem}

\begin{proof}[Proof of Proposition~\ref{190618prop6} assuming Theorem~\ref{thm:slab-dec}]
By affine scaling, we may assume $\alpha_{0} = [0,1] \times [0,\sigma] \times [0,1]$.
By Fubini's theorem, it suffices to show the fiberwise inequality
\begin{equation}
\label{eq:21}
\norm[\Big]{\sum_{\alpha\in \mc{P}^{(2, 1, 2)}(\alpha_0, \sigma)} g_{\alpha}(\cdot,x_{2},\cdot,\cdot,\cdot)}_{L^{4}(\R^{4})}
\lesim_{\epsilon'}
\sigma^{-1-\epsilon'}\Big(\sum_{\alpha}\norm{g_{\alpha}(\cdot,x_{2},\cdot,\cdot,\cdot)}_{L^{4}(\R^{4})}^4 \Big)^{1/4},
\end{equation}
uniformly in $x_{2}$.
This follows from Theorem~\ref{thm:slab-dec} with $\delta=\sigma^{2}$, because for each fixed $x_{2}$ the Fourier support of
\begin{equation}
g_{\alpha}(\cdot,x_{2},\cdot,\cdot,\cdot)
\end{equation}
is contained in the projection of the Fourier support of $g_{\alpha}$ modulo the second coordinate, and this projection satisfies an inclusion of the form \eqref{eq:2} because $s^2 \leq \sigma^2$ when $s \in [0,\sigma]$.
\end{proof}

\begin{remark} \label{rmk5.2}
The above calculation of projection explains the condition $|\sigma''| \leq \sigma^2$ on the right hand side of \eqref{eq:11}, and hence the condition $|\delta''| \leq \delta$ in \eqref{eq:2}, in the same spirit as in Remark \ref{rmk5.1}. The importance of the condition $|\sigma'| \leq \sigma^2$ on the right hand side of \eqref{eq:11} and the condition $|\delta'| \leq \delta^2$ in \eqref{eq:2} can be seen in the proof of Theorem~\ref{thm:slab-dec}, as will be explained in Remark~\ref{rmk5.3} below.
\end{remark}

\subsection{Proof of decoupling on thin slabs}
Theorem~\ref{thm:slab-dec} will follow from

\begin{proposition}\label{190707prop8}
For each $\epsilon>0$,
there exists $K>0$ such that for any $\delta \in (0,1)$
\begin{equation}
\mc{V}(\delta) \leq K^{1/2+\epsilon}\mc{V}(K\delta) + C_K\delta^{-1/2},
\end{equation}
where $C_K$ is a constant depending only on $K$.
\end{proposition}
\begin{proof}[Proof of Theorem~\ref{thm:slab-dec}]
By iterating the result in Proposition~\ref{190707prop8}, $\frac{(\epsilon-1) \log \delta }{\log K}$-many times (assuming without loss of generality that this is a positive integer), we obtain
\begin{equation}
\mc{V}(\delta) \leq \delta^{(\epsilon-1)\frac12-\epsilon} \mc{V}(\delta^{\epsilon})+\widetilde{C}_K  (\log \delta^{-1} )\delta^{-1/2}.
\end{equation}
It remains to note that
\begin{equation}
\mc{V}(\delta^{\epsilon}) \lesssim \delta^{-2\epsilon}
\end{equation}
by the triangle inequality and H\"older's inequality.
\end{proof}

It remains to prove Proposition~\ref{190707prop8}.
We will apply a bilinear method, together with a Bourgain--Guth type argument \cite{MR2860188}.
We need the following bilinear estimate.
\begin{lemma}\label{190707lem9}
Let $K^{-1}>\delta>0$.
Let $j_1, j_2 \in \Z$ with $\abs{j_1-j_2} \geq 2$.
Let $R_{1}=[j_1 K^{-1},(j_1+1)K^{-1}] \times [0,1] \subset [0,1]^2$ and $R_{2}=[j _2 K^{-1},(j_2+1)K^{-1}] \times [0,1] \subset [0,1]^2$.
Then
\begin{equation}
\norm[\bigg]{ \abs[\bigg]{\biggl( \sum_{\beta \in \Part[R_1]{\delta}} g_\beta \biggr) \biggl( \sum_{\beta' \in \Part[R_2]{\delta}} g_{\beta'} \biggr)}^{\frac12} }_{L^4}
\leq C_{K} \delta^{-1/2}
\biggl(
\sum_{ \beta \in \Part{\delta} } \norm{ g_{\beta} }_{L^4}^4 \biggr)^{\frac14}.
\end{equation}
\end{lemma}
Lemma~\ref{190707lem9} would in fact still work under the more relaxed Fourier support assumption $\abs{\delta'} \leq \delta$ in \eqref{eq:2}.
\begin{proof}[Proof of Lemma~\ref{190707lem9}.]
By Plancherel's theorem,
\begin{equation}\label{eq:1}
\norm[\bigg]{ \abs[\bigg]{\biggl( \sum_{\beta \in \Part[R_1]{\delta}} g_\beta \biggr) \biggl( \sum_{\beta' \in \Part[R_2]{\delta}} g_{\beta'} \biggr)}^{\frac12} }_{L^4}
=
\norm[\bigg]{ \sum_{\beta \in \Part[R_1]{\delta}} \sum_{\beta' \in \Part[R_2]{\delta}} \widehat{g_\beta} * \widehat{g_{\beta'}} }_{L^2}^{\frac12}.
\end{equation}
We claim that the collection
\begin{equation}
\label{eq:bdd-overlap}
\Set{ \supp(\widehat{g_{\beta}})+\supp(\widehat{g_{\beta'}}) }_{\beta \in \Part[R_1]{\delta}, \beta' \in \Part[R_2]{\delta}}
\end{equation}
has overlap bounded by a constant depending on $K$: if $\beta_1, \beta_3 \in \Part[R_1]{\delta}$ and $\beta_2, \beta_4 \in \Part[R_2]{\delta}$, and $(r_i,t_i) \in \beta_i$ for $i = 1,2,3,4$ are such that 
\begin{equation}
\label{eq:22}
(r_1,t_1,r_1^2,r_1 t_1)+(r_2,t_2,r_2^2,r_2 t_2) = (r_3,t_3,r_3^2,r_3 t_3)+(r_4,t_4,r_4^2,r_4 t_4) + O(\delta),
\end{equation}
then the distances between $\beta_i$ and $\beta_{i+2}$ are $O(K^2\delta)$ for $i = 1,2$.

The geometric reason for this is that the pieces of the surface $(r,t,r^{2},rt)$ with $(r,t)$ restricted to $R_{1}$ and $R_{2}$, respectively, are transverse.
Indeed, if $(r_{1},t_{1}) \in R_{1}$ and $(r_{2},t_{2})\in R_{2}$, then for bases of tangent spaces at these points we have
\begin{equation} \label{eq:det}
\det
\begin{pmatrix}
1 & 0 & 1 & 0\\
0 & 1 & 0 & 1\\
2r_{1} & 0 & 2r_{2} & 0\\
t_{1} & r_{1} & t_{2} & r_{2}
\end{pmatrix}
= 2(r_{1}-r_{2})^{2} \geq 2K^{-2}.
\end{equation}

However, it is formally easier to verify the bounded overlap property of the collection \eqref{eq:bdd-overlap} algebraically.
Suppose that \eqref{eq:22} holds.
Looking at the third component of \eqref{eq:22}, we obtain
\begin{equation}
\begin{split}
O(\delta)
&=
r_{1}^{2}-r_{3}^{2} + r_{2}^{2}-r_{4}^{2}
=
(r_{1}-r_{3})(r_{1}+r_{3}) + (r_{2}-r_{4})(r_{2}+r_{4})
\\ &=
(r_{4}-r_{2})(r_{1}+r_{3}) + (r_{2}-r_{4})(r_{2}+r_{4}) + O(\delta)
\\ &=
(r_{2}-r_{4})(r_{2}+r_{4}-r_{1}-r_{3}) + O(\delta).
\end{split}
\end{equation}
Since $\abs{r_{2}+r_{4}-r_{1}-r_{3}} \gtrsim 1/K$, it follows that $r_{2}-r_{4} = O(K\delta)$.
Similarly, $r_{1}-r_{3} = O(K\delta)$.
Looking at the fourth component of \eqref{eq:22}, we obtain
\begin{equation}
\begin{split}
O(\delta)
&=
r_{1}t_{1}-r_{3}t_{3} + r_{2}t_{2}-r_{4}t_{4}
=
r_{1}(t_{1}-t_{3}) + r_{2}(t_{2}-t_{4}) + O(K\delta)
\\ &=
r_{1}(t_{4}-t_{2}) + r_{2}(t_{2}-t_{4}) + O(K\delta)
\\ &=
(r_{2}-r_{1})(t_{2}-t_{4}) + O(K\delta).
\end{split}
\end{equation}
Since $\abs{r_{2}-r_{1}} \gtrsim 1/K$, it follows that $t_{2}-t_{4} = O(K^2\delta)$.
Similarly, $t_{1}-t_{3} = O(K^2\delta)$.
This shows that the collection \eqref{eq:bdd-overlap} has bounded overlap.
Therefore,
\begin{equation}
\eqref{eq:1} \leq C_K
\biggl(
\sum_{\beta \in \Part[R_1]{\delta}} \sum_{\beta' \in \Part[R_2]{\delta}}
\norm{ \widehat{g_\beta} * \widehat{g_{\beta'}} }_{L^2}^2 \biggr)^{\frac14}.
\end{equation}
By Plancherel's theorem and H\"{o}lder's inequality, this implies
\begin{equation}
\begin{split}
\eqref{eq:1}
&\leq C_K
\biggl( \sum_{\beta \in \Part[R_1]{\delta}} \sum_{\beta' \in \Part[R_2]{\delta}}
\norm{ g_\beta g_{\beta'} }_{2}^2 \biggr)^{\frac14}
\\ &\leq C_K
\biggl( \sum_{\beta \in \Part[R_1]{\delta}} \norm{g_{\beta}}_{4}^2 \biggr)^{\frac14}
\biggl( \sum_{\beta' \in \Part[R_2]{\delta}} \norm{ g_{\beta'} }_{4}^2 \biggr)^{\frac14}
\\ &\leq C_K
\delta^{-2(\frac12-\frac14)}
\biggl( \sum_{\beta \in \Part[R_1]{\delta}} \norm{g_{\beta}}_{L^4}^4 \biggr)^{\frac18}
\biggl( \sum_{ \beta' \in \Part[R_2]{\delta} } \norm{ g_{\beta'} }_{4}^4 \biggr)^{\frac18}.
\end{split}
\end{equation}
This completes the proof of Lemma~\ref{190707lem9}.
\end{proof}

\begin{proof}[Proof of Proposition~\ref{190707prop8}]
Let $K>0$ be a large number that is to be determined.
For each $j \in \mathbb{Z}$ with $0 \leq j \leq K-1$, we define the strip $R_j =[jK^{-1},(j+1)K^{-1}] \times [0,1]$.
We define
\begin{equation}
G_{j} := \sum_{ \alpha \in \Part[R_j]{\delta}} g_{\alpha}.
\end{equation}
For each $x \in \R^{4}$, we define the significant part
\begin{equation}
\calC(x) := \Set[\Big]{ j' \in \Set{0,\dotsc,K-1} \given \abs[\Big]{ \sum_{j=0}^{K-1} G_{j}(x)} < 10K \abs{G_{j'}(x)} }.
\end{equation}
Note that $\calC(x) \neq \emptyset$ unless $\sum_{j} G_j(x) = 0$.
By considering two possible cases $\abs{\calC(x)} \geq 3$ and $\abs{\calC(x)} \leq 2$, we obtain
\begin{equation} 
\label{eq:BG}
\abs[\Big]{\sum_{\alpha \in \Part[\unitsquare]{\delta}} g_{\alpha}(x)} \leq 10\max_{j} \abs{G_{j}(x)} + 10K \max_{j, j' : \abs{j-j'} \geq 2} \abs{G_{j}(x) G_{j'}(x)}^{\frac12};
\end{equation}
indeed, for $j' \notin \calC(x)$, we have $\abs{G_{j'}(x)} \leq \frac{1}{10K} \abs[\Big]{ \sum_{j=0}^{K-1} G_{j}(x)},$
so 
\begin{equation}
\abs[\Big]{\sum_{j' \notin \calC(x)} G_{j'}(x)} \leq \tfrac{1}{10} \abs[\Big]{\sum_{j' \notin \calC(x)} G_{j'}(x)} + \tfrac{1}{10} \sum_{j' \in \calC(x)} \abs{G_{j'}(x)},
\end{equation}
which implies, if $\abs{\calC(x)} \leq 2$, that
$
\abs[\Big]{\sum_{j' \notin \calC(x)} G_{j'}(x)} \leq \tfrac{2}{9}\max_{j} \abs{G_{j}(x)} 
$
and hence
\begin{equation}
\abs[\Big]{\sum_{j'=0}^{K-1} G_{j'}(x)} \leq \left(2+\tfrac{2}{9} \right) \max_{j} \abs{G_{j}(x)}.
\end{equation}
Raising both sides of \eqref{eq:BG} to the fourth power and integrating in $x$, we obtain
\begin{equation}
\norm[\big]{\sum_{\alpha \in \Part[\unitsquare]{\delta}} g_{\alpha} }_4^4 \lesssim \sum_{j=1}^{K} \norm{G_j}_{4}^4+ K^{4}
\sum_{j, j' : \abs{j-j'} \geq 2}\norm{\abs{G_j G_{j'}}^{\frac12}}_4^4.
\end{equation}
By Lemma~\ref{190707lem9}, the second term is bounded by
\begin{equation}
C_K^4 \delta^{-2}
\biggl( \sum_{\alpha \in \Part[\unitsquare]{\delta}} \norm{ g_{\alpha} }_4^4 \biggr).
\end{equation}
In order to conclude the proof, it suffices to show that, for each $j=0,\dotsc,K-1$, we have
\begin{equation}
\label{eq:12}
\norm{G_j}_{4}
\leq
C K^{1/2} \mc{V}(K\delta)
\biggl( \sum_{\alpha \in \Part[R_{j}]{\delta}} \norm{g_{\alpha}}_{4}^4 \biggr)^{\frac14}
\end{equation}
and take $K$ large enough so that $C \leq K^{\epsilon}$.

By an affine transformation, we may assume without loss of generality that $j=0$.
We define the scalings:
\begin{gather*}
L: (\xi_1,\xi_2,\xi_3,\xi_4) \mapsto (\tfrac{\xi_1}{K},\xi_2,\tfrac{\xi_3}{K^2},\tfrac{\xi_4}{K} ),
\\
L': (\xi_{1},\xi_2 ) \mapsto (\tfrac{\xi_1}{K},\xi_2).
\end{gather*}
Note that $L'$ scales $[0,1]^2$ to the strip $R_0$. The map $L$ is chosen so that $L'$ is the restriction of $L$ to the first two coordinates, and so that $L$ leaves the surface $(r,t,r^{2},rt)$ invariant.

For each $\beta \in \Part[\unitsquare]{K\delta}$, we define a function $H_{\beta}$ by
\begin{equation}
\widehat{H_\beta}(\xi_1,\xi_2,\xi_3,\xi_4) :=
\sum_{\alpha \in \Part[L'(\beta)]{\delta}} \widehat{g_{\alpha}}(L(\xi_1,\xi_2,\xi_3,\xi_4)).
\end{equation}
Then,
\begin{equation} \label{eq:Fsupp_Hbeta}
\supp (\widehat{H_\beta})
\subseteq
\Set{ (r,t,r^2+\delta',rt+\delta'') \given (r,t) \in \beta,\, \abs{\delta'} \leq (K\delta)^2,\, \abs{\delta''} \leq K\delta }.
\end{equation}
Thus, by the definition of the constant $\mc{V}(K\delta)$,
\begin{equation}
\norm[\Big]{ \sum_{\beta \in \Part[\unitsquare]{K\delta}} H_{\beta} }_4
\leq \mc{V}(K\delta)
\biggl( \sum_{\beta \in \Part[\unitsquare]{K\delta}}\norm{H_{\beta}}_4^4 \biggr)^{\frac14}.
\end{equation}
Changing back to the original variables and applying flat decoupling \eqref{eq:flat-dec}, we obtain
\begin{equation}
\begin{split}
\norm{G_j }_4
&\lesssim
\mc{V}(K\delta) \biggl(
\sum_{\beta \in \Part[\unitsquare]{K\delta}}
\norm[\Big]{
\sum_{\alpha \in \Part[L'(\beta)]{\delta}}
g_{\alpha} }_4^4 \biggr)^{\frac14}
\\ &\lesssim K^{1/2}\mc{V}(K\delta)
\biggl(\sum_{\alpha \in \Part[R_j]{\delta}} \norm{g_{\alpha} }_4^4 \biggr)^{\frac14},
\end{split}
\end{equation}
where we used that $\abs{\Part[L'(\beta)]{\delta}} = K$ in the last step.
This finishes the proof of \eqref{eq:12}.
\end{proof}

\begin{remark} \label{rmk5.3}
We have already seen, at the beginning of Section~\ref{subsect:slab}, that one could not replace the condition $|\delta'| \leq \delta^2$ on the right hand side of \eqref{eq:2} by $|\delta'| \leq \delta$ and still hope to prove Theorem~\ref{thm:slab-dec}.
It may be helpful to see what goes wrong in the proof of Theorem~\ref{thm:slab-dec} if we had the condition $|\delta'| \leq \delta$ instead.
In that case, when we rescale, on the right hand side of \eqref{eq:Fsupp_Hbeta}, we would only get $|\delta'| \leq K^2 \delta$, which would not allow us to close the induction on scale argument.
This signifies the advantage to decouple on thin slabs as in Theorem~\ref{thm:slab-dec}.
\end{remark}

\section{Lower bounds}
\label{sec:lower-bd}
In this section, we prove lower bounds for decoupling constants defined in \eqref{eq:dec-const}.
In the case when $p$ is an even integer, such bounds were proved for the related problem of bounding multidimensional Vinogradov mean values in \cite[Theorem 3.1]{PPW}.
However, the construction given there does not detect the optimality of the bound \eqref{r2s2rt_dec_const} for $p \in (4,6)$.

\begin{proposition}
\label{prop:skew-lower-bd}
Let $Q_{1}(\bft),\dotsc,Q_{n}(\bft)$ be quadratic forms in $d$ variables.
Suppose that $Q_{1},\dotsc,Q_{n'}$ do not depend on $t_{d'+1},\dotsc,t_{d}$ for some partitions $n=n'+n''$ and $d=d'+d''$.
Let $\calK'' := d''+2n''$.

Let $\calS$ be the surface defined in \eqref{eq:S0} and let $\Dec_{\calS}(\delta,p)$ be the associated decoupling constant, defined in \eqref{eq:dec-const}.
Then, for $2 \leq p < \infty$, we have
\begin{equation}
\label{eq:skew-lower-bd}
\Dec_{\calS}(\delta,p) \gtrsim \delta^{-d'(1/2-1/p)} \cdot \delta^{-d''(1-1/p)+\calK''/p}.
\end{equation}
\end{proposition}

We postpone the proof of Proposition~\ref{prop:skew-lower-bd} till the end of this section and indicate how it recovers the lower bounds in \cite[Section 1.5]{MR4143735}.

\begin{corollary}
\label{cor:dec-lower-bd}
Let $Q_{1},\dotsc,Q_{n}$ be quadratic forms in $d$ variables and $\calK := d+2n$.
Let $\calS$ be the surface defined in \eqref{eq:S0} and let $\Dec_{\calS}(\delta,p)$ be the associated decoupling constant, defined in \eqref{eq:dec-const}.
Then, for $2 \leq p < \infty$, we have
\begin{equation}
\label{eq:dec-lower-bd}
\Dec_{\calS}(\delta,p) \gtrsim \max( \delta^{-d(1/2-1/p)}, \delta^{-d(1-1/p)+\calK/p} ).
\end{equation}
\end{corollary}

\begin{proof}[Proof of Corollary~\ref{cor:dec-lower-bd} assuming Proposition~\ref{prop:skew-lower-bd}]
The hypothesis of Proposition~\ref{prop:skew-lower-bd} clearly holds for arbitrary quadratic forms $Q_{1},\dotsc,Q_{n}$ with either $d'=n'=0$ or $d''=n''=0$.
\end{proof}

By considering functions $f_{\theta}$ of tensor product form, we also obtain the following lower bound.
\begin{lemma}
\label{lem:subspace-lower-bd}
In the situation of Corollary~\ref{cor:dec-lower-bd}, if $V \leq \R^{d}$ is a linear subspace, $\tilde{Q}_{j} := Q_{j}|_{V}$, $j=1,\dotsc,n$, are restrictions of $Q_{j}$'s to $V$, and $\tilde{\calS}$ is the surface
\begin{equation}
\Set{ (\bft, \tilde{Q}_{1}(\bft), \dotsc, \tilde{Q}_{n}(\bft)) \given \abs{\bft} \leq 1 },
\end{equation}
then
\begin{equation}
\label{eq:subspace-lower-bd}
\Dec_{\calS}(\delta,p) \gtrsim \Dec_{\tilde{\calS}}(\delta,p).
\end{equation}
\end{lemma}

Now we can justify the lower bounds on the sharp decoupling exponents in Theorem~\ref{thm:classification+exponents}.

For surfaces \eqref{S_def}, we have $d=3$ and $\calK=7$.
Hence, \eqref{eq:dec-lower-bd} above implies the lower bounds
\begin{equation}
\label{eq:16}
\Dec_{\calS}(\delta,p) \gtrsim \max( \delta^{-3(1/2-1/p)}, \delta^{-3+10/p}).
\end{equation}
This shows the lower bounds on sharp decoupling exponents in \eqref{non-deg-decexp}.

In case~\ref{it:class:common-factor} of Theorem~\ref{thm:classification+exponents}, assume without loss of generality that $(P,Q)=(rs,rt)$.
Then we apply \eqref{eq:subspace-lower-bd} for the subspace given by $r=0$.
On this subspace, we apply \eqref{eq:dec-lower-bd} with $\tilde{d}=2$ and $\tilde{\calK}=2$.
This gives the additional lower bound
\begin{equation}
\label{eq:20}
\Dec_{\calS}(\delta,p) \gtrsim \delta^{-2(1-1/p)+2/p} = \delta^{-2+4/p}.
\end{equation}

In case~\ref{it:complete-square} of Theorem~\ref{thm:classification+exponents}, assume without loss of generality that $P=r^{2}$.
Applying \eqref{eq:skew-lower-bd} with $d'=1$ and $n'=1$, we obtain
\begin{equation}
\Dec_{\calS}(\delta,p) \gtrsim \delta^{-1\cdot (1/2-1/p)} \delta^{-2\cdot(1-1/p)+4/p} = \delta^{-5/2+7/p}.
\end{equation}
This shows the middle lower bound in \eqref{r2s2rt_decexp}.

\begin{proof}[Proof of Proposition~\ref{prop:skew-lower-bd}]
Write points in $\R^{d+n}$ as $(x',x'',y',y'') \in \R^{d'+d''+n'+n''}$.
For $\theta \in \Part{\delta}$, write $\theta = \theta' \times \theta''$.
Choose functions $f_{\theta}$ of the form $f_{\theta} = g_{\theta'}(x',y')h_{\theta}(x'',y'')$ with the following properties
\begin{enumerate}
\item $\widehat{g_{\theta'}}$ and $\widehat{h_{\theta}}$ are positive smooth functions,
\item $\int \widehat{g_{\theta'}} = \int \widehat{h_{\theta}} = 1$,
\item $\widehat{g_{\theta'}}$ is supported in a ball of radius $\approx \delta^2$,
\item $\widehat{f_{\theta}}$ is supported in and adapted to a rectangular box of dimensions
\begin{equation}
\underbrace{\delta^{2}/10 \times \dotsm \times \delta^{2}/10}_{\text{$d'$ times}} \times
\underbrace{\delta/10 \times \dotsm \times \delta/10}_{\text{$d''$ times}} \times
\underbrace{\delta^{2}/10 \times \dotsm \times \delta^{2}/10}_{\text{$n$ times}}
\end{equation}
inside the set \eqref{eq:fourier-supp}.
\end{enumerate}
Note that $g_{\theta'}$ depends only on the projection of $\theta$ onto $\R^{d'+n'}$, whereas $h_{\theta}$ has to depend on $\theta$ because of the geometry of the set \eqref{eq:fourier-supp}.

On one hand, $\norm{f_{\theta}}_{p} \sim \delta^{-(2d'+d''+2n)/p}$, and by definition we have
\begin{equation}
\norm[\big]{ \sum_{\theta\in\Part{\delta}} f_{\theta} }_{p}
\leq
\Dec_{\calS}(\delta,p) \Bigl( \sum_{\theta\in\Part{\delta}} \norm{f_{\theta}}_{p}^{p} \Bigr)^{1/p}
\sim
\Dec_{\calS}(\delta,p) \delta^{-d/p} \delta^{-(2d'+d''+2n)/p}.
\end{equation}
On the other hand,
\begin{equation}
\begin{split}
\norm[\big]{ \sum_{\theta\in\Part{\delta}} f_{\theta} }_{p}
&\gtrsim
\inf_{\substack{x''\in\R^{d''},y''\in\R^{n''},\\\abs{x''},\abs{y''} \leq 1/100}} \norm[\big]{ \sum_{\theta\in\Part{\delta}} f_{\theta} }_{L^{p}(\R^{d'} \times \Set{x''} \times \R^{n'} \times \Set{y''})}
\\ &=
\inf_{\substack{x''\in\R^{d''},y''\in\R^{n''},\\\abs{x''},\abs{y''} \leq 1/100}} \norm[\big]{ \sum_{\theta'} c_{\theta',x'',y''} g_{\theta'} }_{L^{p}(\R^{d'} \times \R^{n'})}
\end{split}
\end{equation}
where $c_{\theta',x'',y''} := \sum_{\theta''} h_{\theta}(x'',y'')$ is independent of $x',y'$ and satisfies 
\begin{equation}
\abs{ c_{\theta',x'',y''} } \sim \delta^{-d''}
\end{equation}
uniformly in $\theta'$ and $|x''|, |y''| \leq 1/100$. This is because there is almost no cancellation in the sum over $\theta''$. 

Let $\phi = \eta(\delta^{2}\cdot)$, where $\eta$ is a fixed positive Schwartz function on $\R^{d'}\times\R^{n'}$ with $\supp \widehat{\eta} \subset B(0,1/10)$.
Then, by H\"older's inequality,
\begin{equation}
\begin{split}
\norm[\big]{ \sum_{\theta'} c_{\theta',x'',y''} g_{\theta'} }_{L^{p}(\R^{d'}\times\R^{n'})}
&\geq
\norm{\phi}_{1/(1/2-1/p)}^{-1}
\norm[\big]{ \phi \sum_{\theta'} c_{\theta',x'',y''} g_{\theta'} }_{L^{2}(\R^{d'}\times\R^{n'})}
\\ &\sim
\delta^{2\cdot (d'+n')(1/2-1/p)} \norm[\big]{ \sum_{\theta'} c_{\theta',x'',y''} \phi g_{\theta'} }_{L^{2}(\R^{d'}\times\R^{n'})}.
\end{split}
\end{equation}
Since the Fourier supports of $\phi g_{\theta'}$  are disjoint for different $(\theta')$'s, we obtain
\begin{equation} 
\label{eq:orthogonality_in_theta'}
\begin{split}
\norm[\big]{ \sum_{\theta'} c_{\theta',x'',y''} \phi g_{\theta'} }_{L^{2}(\R^{d'}\times\R^{n'})}
&=
\Bigl( \sum_{\theta'} \abs{ c_{\theta',x'',y''} }^2 \norm[\big]{ \phi g_{\theta'} }_{L^{2}(\R^{d'}\times\R^{n'})}^{2} \Bigr)^{1/2} \\
&\sim 
\delta^{-d'/2} \cdot \delta^{-d''} \cdot \delta^{-2 \cdot (d'+n')/2},
\end{split}
\end{equation}
uniformly for $|x'|, |y'| \leq 1/100$.
Combining the above estimates, we obtain
\begin{equation}
\Dec_{\calS}(\delta,p) \delta^{-d/p} \delta^{-(2d'+d''+2n)/p}
\gtrsim
\delta^{2\cdot (d'+n')(1/2-1/p)}
\cdot
\delta^{-d'/2} \cdot \delta^{-d''} \cdot \delta^{-2 \cdot (d'+n')/2}.
\end{equation}
This implies the claim \eqref{eq:skew-lower-bd}.
\end{proof}

\appendix
\section{Facts about decoupling inequalities}

\subsection{Interpolation}
\label{sec:interpolation}
It is well-known that decoupling inequalities can be interpolated by the argument in \cite[Proposition 6.2]{MR3374964}.
However, as observed e.g.\ in \cite[Appendix B]{GZo18}, a simpler argument is available when decoupling inequalities are stated for functions satisfying the relaxed Fourier support condition \eqref{eq:fourier-supp}.
We record this simpler argument in this appendix.

Let $\delta \in (0,1]$ be a dyadic number.
For every dyadic cube $\Box \in \Part{\delta}$, let $\calU_{\Box}$ be the smallest rectangular box that contains the uncertainty region \eqref{eq:fourier-supp}.
Then, using the decoupling inequality \eqref{eq:dec-const} at a scale slightly larger than $\delta$, we see that the estimate
\begin{equation}
\label{eq:13}
\norm[\Big]{\sum_{\Box \in \Part{\delta}} f_{\Box} }_{L^p(\R^{d+n})}
\lesssim \Dec_{\mc{S}_0}(C\delta, p)
\Big(\sum_{\Box \in \Part{\delta}} \norm{f_{\Box}}_{L^p(\R^{d+n})}^{p} \Big)^{1/p}
\end{equation}
continues to hold for arbitrary functions $f_{\Box}$ with $\supp \widehat{f_{\Box}} \subseteq 2\calU_{\Box}$, where $2\calU_{\Box}$ are rectangular boxes with the same center and orientation but twice the side lengths.
Let $\psi_{\Box}$ be Schwartz functions such that $\one_{\calU_{\Box}} \leq \widehat{\psi}_{\Box} \leq \one_{2\calU_{\Box}}$ that are scaled, rotated, and modulated copies of a fixed Schwartz function.
Then it follows that
\begin{equation}
\label{eq:14}
\begin{split}
\norm[\Big]{\sum_{\Box \in \Part{\delta}} \psi_{\Box} * g_{\Box} }_{L^p(\R^{d+n})}
&\lesssim \Dec_{\mc{S}_0}(C\delta, p)
\Big(\sum_{\Box \in \Part{\delta}} \norm{\psi_{\Box} * g_{\Box}}_{L^p(\R^{d+n})}^{p} \Big)^{1/p}
\\ &\lesssim \Dec_{\mc{S}_0}(C\delta, p)
\Big(\sum_{\Box \in \Part{\delta}} \norm{g_{\Box}}_{L^p(\R^{d+n})}^{p} \Big)^{1/p}
\end{split}
\end{equation}
for \emph{arbitrary} functions $g_{\Box}$, where the first inequality holds by \eqref{eq:13} and the second by the Young convolution inequality.
Since the estimate \eqref{eq:14} holds for arbitrary functions $g_{\Box} \in L^{p}(\R^{d+n})$, we may use the complex interpolation theorem for linear operators on $L^{p}$ spaces to conclude that
\begin{equation}
\label{eq:dec-interpolation}
\Dec(\delta,p_{\theta})
\lesssim
\Dec(C\delta,p_{0})^{1-\theta} \Dec(C\delta,p_{1})^{\theta},
\end{equation}
with the usual conventions
\begin{equation}
\tfrac{1}{p_{\theta}} = \tfrac{1-\theta}{p_{0}} + \tfrac{\theta}{p_{1}},
\quad
0 \leq \theta \leq 1,
\end{equation}
for any $2 \leq p_{0},p_{1} \leq \infty$.

As a consequence, $\alpha \mapsto \gamma_{P,Q}(1/\alpha)$ is a convex function on $[0,1/2]$.

\subsection{Flat decoupling}
The same argument as above shows that, for an arbitrary collection $U$ of parallelepipeds $\calU$ in $\R^{d}$ such that $2\calU$ have bounded overlap, we have
\begin{equation}
\label{eq:flat-dec}
\norm[\Big]{\sum_{\calU\in U} f_{\calU} }_{L^p(\R^{d})}
\lesssim \abs{U}^{1-2/p}
\Big(\sum_{\calU \in U} \norm{f_{\calU}}_{L^p(\R^{d})}^{p} \Big)^{1/p}
\end{equation}
for arbitrary functions $f_{\calU}$ with $\supp \widehat{f_{\calU}} \subseteq \calU$.
The inequality \eqref{eq:flat-dec} is called ``flat decoupling'' because it does not require any curvature assumption on the Fourier supports. This inequality was already observed and applied in Tao and Vargas \cite{MR1748920,MR1748921}. 

\printbibliography

\label{endofpaper}
\end{document}